\documentclass[reqno,12pt]{amsart}

\usepackage{amscd,amsfonts,mathrsfs,amsthm,enumerate}
\usepackage{amssymb, amsmath}
\usepackage{stmaryrd}
\usepackage{epsfig}
\usepackage[alphabetic, sorted-cites]{amsrefs}
\usepackage{color}

\def\real     #1{{\mathbb R^{#1}}}
\def\complex  #1{{\mathbb C^{#1}}}
\def\natural  #1{{\mathbb N^{#1}}}

\def\dt       {\partial_{t}}

\def\equationcolor {\color{black}}
\def\textcolor     {\color{black}}

\def\bcoleq    {\begin{equation}\equationcolor}
\def\ecoleq    {\textcolor\end{equation}}
\def\bcoleqn   {\equationcolor\begin{eqnarray}}
\def\ecoleqn   {\end{eqnarray}\textcolor}

\def\gm{{\operatorname{g}_M}}
\def\gn{{\operatorname{g}_N}}

\def\gk{{\operatorname{g}_{M\times N}}}

\def\rn{{\operatorname{R}_N}}
\def\rk{{\tilde{R}}}

\def\rind{\operatorname{R}}

\def\dF{\operatorname{d}\hspace{-3pt}F}
\def\df{\operatorname{d}\hspace{-3pt}f}

\def\gind{\operatorname{g}}

\DeclareMathOperator*{\Id}{\operatorname{Id}}

\newtheorem{theorem}{Theorem}[section]
\newtheorem{mythm}{Theorem}

\newtheorem{lemma}[theorem]{Lemma}

\newtheorem{proposition}[theorem]{Proposition}
\newtheorem{definition}[theorem]{Definition}
\theoremstyle{definition}
\newtheorem{remark}[theorem]{Remark}

\newcommand{\bfig}{\begin{figure}}
\newcommand{\efig}{\end{figure}}

\makeatletter
\def\pproof#1{\@ifnextchar[\opargproof
{\opargproof[\it Proof of #1.]}}
\def\opargproof[#1]{\par\noindent {\bf #1 }}

\makeatother

\numberwithin{equation}{section}

\begin{document}

\title[Deformation of maps]{Mean curvature flow of area decreasing maps between Riemann surfaces}
\author[Andreas Savas-Halilaj]{\textsc{Andreas Savas-Halilaj}}
\author[Knut Smoczyk]{\textsc{Knut Smoczyk}}

\renewcommand{\subjclassname}{  \textup{2000} Mathematics Subject Classification}
\thanks{The authors are supported by the grant DFG SM 78/6-1.}

\address{Andreas Savas-Halilaj\newline
Institut f\"ur Differentialgeometrie\newline
Leibniz Universit\"at Hannover\newline
Welfengarten 1\newline
30167 Hannover, Germany\newline
{\sl E-mail address:} {\bf savasha@math.uni-hannover.de}
}
\address{Knut Smoczyk\newline
Institut f\"ur Differentialgeometrie and\newline
Riemann Center for Geometry and Physics\newline
Leibniz Universit\"at Hannover\newline
Welfengarten 1\newline
30167 Hannover, Germany\newline
{\sl E-mail address:} {\bf smoczyk@math.uni-hannover.de}
}

\date{}

\begin{abstract}
In this article we give a complete description of the evolution of an area
decreasing map $f:M\to N$ induced by its mean curvature in the situation where $M$ and $N$ are complete Riemann surfaces with bounded geometry,
$M$ being compact, for which their sectional curvatures $\sigma_M$, $\sigma_N$ satisfy $\min\sigma_M\ge\sup\sigma_N$.
\end{abstract}

\maketitle
\setcounter{tocdepth}{1}

\section{Introduction}

Let $(M,\gm)$ and $(N,\gn)$ be two complete Riemann surfaces, with $(M,\gm)$ being compact. A smooth map
$f:M\to N$ is called {\it area decreasing} if $|\operatorname{Jac}(f)|\le 1$, where $\operatorname{Jac}(f)$ is the
{\it Jacobian determinant} of $f$ (for short just {\it Jacobian}). Being area decreasing means that the map $f$ contracts $2$-dimensional regions of $M$. If $|\operatorname{Jac}(f)|<1$
the map is called {\it strictly area decreasing} and if $|\operatorname{Jac}(f)|=1$
the map is said {\it area preserving}. Note that in the latter case $\operatorname{Jac}(f)=\pm 1$ depending on whether $f$ is
orientation preserving or orientation reversing map.
In this article we deform area decreasing maps $f$ by evolving their corresponding graphs
$$\Gamma(f):=\big\{(x,f(x))\in M\times N:x\in M\big\},$$
under the mean curvature flow in the Riemannian product $4$-manifold
$$(M\times N,\gk=\pi^*_M\gm+\pi^*_N\gn),$$
where here
$\pi^*_M:M\times N\to M$ and $\pi^*_N:M\times N\to N$ are the natural projection maps.
Our main goal is to show the
following theorem which generalizes  all the previous known results for area decreasing maps between
Riemann surfaces evolving under the mean curvature flow.

\begin{mythm}\label{thma}
Let $(M,\gm)$ and $(N,\gn)$ be complete Riemann surfaces, $M$ being compact and $N$ having bounded geometry.
Let $f:M\to N$ be a smooth area decreasing map.
Suppose that the sectional curvatures $\sigma_M$ of $\gm$ and $\sigma_N$ of $\gn$ are related by
$\min\sigma_M\ge\sup\sigma_N.$
Then there exists a family of smooth area decreasing maps $f_t:M\to N$, $t\in[0,\infty)$,
$f_0=f$, such that the graphs $\Gamma(f_t)$ of $f_t$ move by mean curvature flow in $(M\times N,\gk)$. Moreover, the family $\Gamma(f_t)$
smoothly converges to a limiting surface $M_{\infty}$ in $M\times N$. Furthermore, there exist only two possible categories of
initial data sets and corresponding solutions:
\begin{enumerate}[\rm (I)]
\item
The curvatures $\sigma_M$, $\sigma_N$ are constant and equal and the map $f_0$ is area preserving.
In this category, each $f_t$ is area preserving and the limiting surface $M_{\infty}$
is a minimal Lagrangian graph in $M\times N$, with respect to the symplectic form
$$\Omega_{M\times N}:=\pi^*_M\Omega_M\mp\pi^*_N\Omega_N,$$
depending on whether $f_0$ is orientation preserving or reversing, respectively. Here $\Omega_N$
and $\Omega_N$ are the positively oriented volume forms of $M$ and $N$, respectively.
\medskip
\item
All other possible cases. In this category, for $t>0$ each map $f_t$ is strictly area decreasing.
Moreover, depending on the sign of $\sigma:=\min\sigma_M$ we have the following behavior:
\smallskip
\begin{enumerate}
\item[\rm (a)]
If $\sigma>0$, then $M_{\infty}$ is the graph of a constant map.
\medskip
\item[\rm (b)]
If $\sigma=0$, then $M_{\infty}$ is a totally geodesic graph of $M\times N$.
\medskip
\item[\rm (c)]
If $\sigma<0$, then $M_{\infty}$ is a minimal surface of $M\times N$.
\end{enumerate}
\end{enumerate}
\end{mythm}

\begin{remark}
Some parts of Theorem \ref{thma}, especially in the case where $\sigma_M$ and $\sigma_N$ are constant, are already known. More precisely:
\begin{enumerate}[(i)]
\item
If the initial data set belongs to category (I), then $N$ is compact because $f_0$ is a local diffeomorphism.
On the other hand, the maps $f_t$
will be area preserving for all $t$ since this is a special case of the Lagrangian mean
curvature flow (see \cite{sm1} or the survey paper \cite{sm}). Now the statement
of category (I) follows from the results in Wang \cite{wang2} and  Smoczyk \cite{smoczyk1}.
\smallskip
\item
If the initial data set belongs to category (II), that is either $f_0$ is not area preserving everywhere or
$\sigma_M=\sigma=\min\sigma_M=\sigma_N$ does not hold at each point, then (as will be shown in Lemma \ref{jacobians})
$f_t$ will be strictly area decreasing for all $t>0$. Then, if $N$ is compact, part (a) of category (II)
was shown in \cite{savas2}.
\smallskip
\item
In the category (IIc), the minimal surface $M_{\infty}$ is not necessarily totally geodesic. One reason
is that there is an abundance of examples of minimal graphs that are generated by area decreasing maps
between two negatively curved compact hyperbolic surfaces. For instance, any holomorphic map
between compact hyperbolic spaces is area decreasing due to the Schwarz-Pick-Yau Lemma
\cite{yau1} and its graph is minimal.
\end{enumerate}
\end{remark}

Another aim of this paper is to obtain curvature estimates for the graphical mean curvature and for the second fundamental
form of the evolving graph. In particular, we prove the following theorem.

\begin{mythm}\label{thmb}
Suppose that $(M,\gm)$ and $(N,\gn)$ are Riemann surfaces as in Theorem \ref{thma} and $f:M\to N$ a smooth strictly area decreasing map.
Suppose that the sectional curvatures $\sigma_M$ of $\gm$ and $\sigma_N$ of $\gn$ are 
related by
$$\sigma:=\min\sigma_M\ge\sup\sigma_N.$$
Then, depending on the sign of the constant $\sigma$, we have the following
decay estimates for the mean curvature flow of the graph of $f$ in the product Riemannian manifold $(M\times N,\gk)$:
\begin{enumerate}[\rm(a)]
\item
If $\sigma>0$, then there exists a uniform time independent constant $C$ such that the norm of the second
fundamental form satisfies
$$\|A\|^2\le{C}{t^{-1}}.$$
\item
If $\sigma=0$, then there exists a uniform time independent constant $C$ such that the norms of the
second fundamental form and of the mean curvature satisfy
$$\quad\quad\|A\|^2\le C,\quad \int_{M}\|A\|^2\Omega_M\le{C}{t}^{-1}\quad\text{and}
\quad \|H\|^2\le C t^{-1}.$$
\item
If $\sigma<0$, then there exists a uniform time independent constant $C$ such that
$$\|A\|^2\le C.$$
\end{enumerate}
\end{mythm}

\begin{remark}
Similar curvature decay estimates for the norm of the second fundamental form in the
case $\sigma>0$,
were obtained also by Lubbe \cite{lubbe}. Explicit curvature decay estimates have been obtained
recently by Smoczyk, Tsui and Wang \cite{stw} in the case of strictly area decreasing 
Lagrangian maps between flat Riemann surfaces.
\end{remark}

\section{Geometry of graphical surfaces}
In this section we recall some basic facts about graphical surfaces. Some of these can be found in our previous papers 
\cite{savas3,savas2,savas1}.
In order to make the paper self-contained let us recall very briefly some of them here.

\subsection{Notation}
Let $F:\Sigma\to L$ be an isometric embedding of an $m$-dimensional 
Riemannian manifold 
$\big(\Sigma,\gind\big)$ to a Riemannian manifold $\big(L,\langle\cdot\,,\cdot\rangle\big)$ of dimension $l$. We denote by ${\nabla}$
the Levi-Civita connection associated to $\gind$ and by
$\tilde{\nabla}$ the corresponding Levi-Civita of
$\langle\cdot\,,\cdot\rangle$. The differential $\dF$ is a section in $F^{\ast}TL\otimes 
T^*\Sigma$. Let $\nabla^F$ be the connection induced by $F$ on this bundle.
The covariant derivative of $\dF$ is called the \textit{second fundamental 
tensor} $A$ of $F$, i.e.,
$$ A(v_1,v_2):=\big(\nabla^{F}\dF\big)
(v_1,v_2)=\tilde\nabla_{\dF(v_1)}\dF(v_2)-\dF\big(\nabla_{v_1}v_2\big),$$
for any $v_1,v_2\in T\Sigma$. Note that the second fundamental form maps
to the normal bundle $\mathcal{N}\Sigma$. The {\it second fundamental form
with respect to a normal direction} $\xi$ is denoted by $A^{\xi}$, that is 
$A^{\xi}(v_1,v_2):=\langle A(v_1,v_2),\xi\rangle$, for any pair $v_1,v_2\in T\Sigma$.
The trace $H$ of $A$ with respect to $\gind$ is
called  the {\it mean curvature vector field} of the graph. If $H$ vanishes 
identically  then the embedding $F$ is called {\it minimal}.

The normal bundle $\mathcal{N}\Sigma$ admits a natural connection
which we denote by $\nabla^{\perp}$. Let us denote by $\rind$, 
$\tilde{\rind}$ and $\rind^{\perp}$ the curvature operators of $T\Sigma$,
$TL$ and $\mathcal{N}\Sigma$, respectively. Then these tensors are related with $A$
through the Gau{\ss}-Codazzi-Ricci equations. Namely:
\begin{enumerate}[\rm (a)]
\item{\bf Gau{\ss} equation}
\begin{eqnarray*}
\quad\,\,\,\rind(v_1,v_2,v_3,v_4)&=&F^{\ast}\tilde{\rind}(v_1,v_2,v_3,v_4)\\
&&+\langle A(v_1,v_3),A(v_2,v_4)\rangle
-\langle A(v_2,v_3),A(v_1,v_4)\rangle,
\end{eqnarray*}
for any $v_1,v_2,v_3,v_4\in T\Sigma$.
\medskip
\item{\bf Codazzi equation}
$$\big(\nabla^{\perp}_{v_1}A\big)(v_2,v_3)-\big(\nabla^{\perp}_{v_2}A\big)(v_1,v_3)
=-\sum_{\alpha=m+1}^{l}\tilde\rind(v_1,v_2,v_3,\xi_{\alpha})\xi_{\alpha},$$
where $v_1,v_2,v_3\in T\Sigma$ and $\{\xi_{m+1},\dots,\xi_{l}\}$ is a local
orthonormal frame field in the normal bundle of $F$.
\medskip
\item{\bf Ricci equation}
\begin{eqnarray*}
&&\rind^{\perp}(v_1,v_2,\xi,\eta)=\tilde\rind\big(\dF(v_1),\dF(v_2),\xi,\eta\big)\\
&&\quad\quad\quad\quad\quad\quad
+\sum_{k=1}^m\big\{A^{\xi}(v_1,e_k)A^{\eta}(v_2,e_k)
-A^{\eta}(v_1,e_k)A^{\xi}(v_2,e_k)\big\},
\end{eqnarray*}
where here $v_1,v_2\in T\Sigma$, $\xi,\eta\in\mathcal{N}\Sigma$ and $\{e_1,\dots,e_m\}$ is a local orthonormal frame field with respect to $\gind$.
\end{enumerate}

\subsection{Graphs}
Suppose now that the manifold $L$ is a product of two Riemann surfaces $(M,\gm)$ and $(N,\gn)$
and that $f:M\to N$ is a smooth map.
The induced metric on the product manifold will 
be denoted by
$$\gk:=\langle\cdot\,,\cdot\rangle=\gm\times \gn.$$
Define the embedding $F:M\to M\times N$, given by
$$F(x):=(\Id\times f)(x)=\bigl(x,f(x)\bigr),$$
for any point $x\in M$. The graph of $f$ is defined to be the submanifold $\Gamma(f):=F(M)$.
Since 
the map $F$ is an embedding,
it induces another Riemannian metric
$\gind:=F^*\gk$
on $M$. Following Schoen\rq{s} \cite{schoen} terminology, we call $f$ a {\it minimal map}
if its graph $\Gamma(f)$ is a minimal submanifold of $M\times N$.
The natural projections
$\pi_{M}:M\times N\to N$, $\pi_{N}:M\times N\to N$
are submersions. Note that the tangent bundle  of the product manifold
$M\times N$, splits as a direct sum
\begin{equation*}
T(M\times N)=TM\oplus TN.
\end{equation*}
The metrics $\gm,\gk$ and $\gind$ are related by
\begin{eqnarray*}
\gk&=&\pi_M^*\gm+\pi_N^*\gn\,,\label{met1}\\
\gind&=&\gm+f^*\gn\,.\label{met2}
\end{eqnarray*}
The Levi-Civita connection $\tilde\nabla$ of the product manifold is
related to the Levi-Civita connections $\nabla^{\gm}$ and $\nabla^{\gn}$ by
$$\tilde\nabla=\pi_M^*\nabla^{\gm}\oplus\pi_N^*\nabla^{\gn}\,.$$
The corresponding curvature operator $\rk$ is related to the curvature
operators $\operatorname{R}_M$ and $\rn$ by
\begin{equation*}
\rk=\pi^{*}_{M}\operatorname{R}_M\oplus\pi^{*}_{N}\rn.
\end{equation*}
The Levi-Civita connection of $\gind$ will be denoted by $\nabla$, its curvature
tensor by $\rind$ and it sectional curvature by $\sigma_{\gind}$. We denote the sectional curvatures of
$(M,\gm)$ and $(N,\gn)$ by $\sigma_M$ and $\sigma_N$, respectively.

\subsection{Singular decomposition}\label{singular}
Let us recall here some basic Linear Algebra constructions. Fix a point $x\in M$. Let $\lambda^2\le\mu^2$ be the eigenvalues of $f^*\gn$ with respect to $\gm$ at $x$ and denote by
$\{\alpha_1,\alpha_2\}$ a positively oriented orthonormal (with respect to $\gm$) basis of eigenvectors.
The corresponding values $0\le\lambda\le\mu$ are called {\it singular values} of $f$ at $x$. Then, there exists an orthonormal (with respect to $\gn$) basis $\{\beta_{1},\beta_{2}\}$ of $T_{f(x)}N$ such that
$$\df(\alpha_{1})=\lambda\beta_{1}\quad\text{and}\quad
\df(\alpha_{2})=\mu\beta_{2}.$$
Indeed, in the case where the values $\lambda$ and $\mu$ are strictly positive, one may define them as
$$\beta_1:=\frac{\df(\alpha_1)}{\|\df(\alpha_1)\|}\quad\text{and}\quad\beta_2:=\frac{\df(\alpha_2)}{\|\df(\alpha_2)\|}.$$
In the case where $\lambda$ vanishes and $\mu$ is positive, define first $\beta_2$ by
$$\beta_2:=\frac{\df(\alpha_2)}{\|\df(\alpha_2)\|}$$
and take as $\beta_1$ a unit vector perpendicular to $\beta_2$. In the special case where both $\lambda$ and $\mu$ are zero, we may take an arbitrary orthonormal basis of $T_{f(x)}N$.
This procedure is called the \textit{singular decomposition}
of the differential $\df$ of the map $f$. Observe that
$$v_1:=\frac{\alpha_1}{\sqrt{1+\lambda^2}}\quad\text{and}\quad v_2:=\frac{\alpha_2}{\sqrt{1+\mu^2}}$$
are orthonormal with respect to the metric $\gind$. Hence, the vectors
\begin{equation*}
e_{1}:=\frac{1}
{\sqrt{1+\lambda^2}}\big(\alpha_1\oplus\lambda\beta_1\big)
\quad\text{and}\quad
e_{2}:=\frac{1}
{\sqrt{1+\mu^2}}\big(\alpha_2\oplus\mu\beta_2\big)\label{tangent}
\end{equation*}
form an orthonormal basis with respect to the metric $\gk$ of the tangent space 
$\dF\left(T_{x}M\right)$ of the graph $\Gamma(f)$ at
$x$. Moreover, the vectors
\begin{equation*}
e_{3}:=\frac{1}
{\sqrt{1+\lambda^2}}\big(-\lambda\alpha_1\oplus\beta_1\big)
\quad\text{and}\quad
e_{4}:=\frac{1}
{\sqrt{1+\mu^2}}\big(-\mu\alpha_2\oplus\beta_2\big)\label{normal}
\end{equation*}
form an orthonormal basis with respect to  $\gk$ of the normal space $\mathcal{N}_{x}M$ of the
graph $\Gamma(f)$ at the point $f(x)$. Observe now that
$$e_1\wedge e_2\wedge e_3\wedge e_4=\alpha_1\wedge\alpha_2\wedge\beta_1\wedge\beta_2.$$
Consequently, $\{e_3,e_4\}$ is an oriented basis of the normal space $\mathcal{N}_xM$ if and only if $\{\alpha_1,\alpha_2,\beta_1,\beta_2\}$ is an oriented basis of $T_xM\times T_{f(x)}N$.

The {\it area functional} $A(f)$ of the graph is given by
$$A(f):=\int_M \sqrt{\det(\gm+f^*\gn)}\,\Omega_M=\int_M \sqrt{(1+\lambda^2)(1+\mu^2)}\,\Omega_M.$$

\subsection{Jacobians of the projection maps}\label{sec 2.4}
As before let $\Omega_M$ denote the K\"ahler form of the Riemann surface $(M,\gm)$ and $\Omega_N$ the K\"ahler form of $(N,\gn)$.
We can extend $\Omega_M$ and $\Omega_N$ to two parallel $2$-forms on the product manifold $M\times N$ by pulling them back via the projection maps $\pi_M$
and $\pi_N$. That is we may define the parallel forms
$$\Omega_1:=\pi^*_M\Omega_M\quad\text{and}\quad\Omega_2:=\pi^*_N\Omega_N.$$
Define now two smooth functions $u_1$ and $u_2$ given by
$$u_1:=\ast (F^*\Omega_1)=\ast\big\{(\pi_M\circ F)^*\Omega_M\big\}=\ast ({\Id}^*\Omega_M) $$
and
$$u_2:=\ast (F^*\Omega_2)=\ast\big\{(\pi_N\circ F)^*\Omega_N\big\}=\ast (f^*\Omega_N) $$
where here $\ast$ stands for the Hodge star operator with respect to the metric $\gind$. Note that $u_1$
is the Jacobian of the projection map from $\Gamma(f)$ to the first factor of $M\times N$ and $u_2$ is the Jacobian
of the projection map of $\Gamma(f)$ to the second factor of $M\times N$.
With respect to the basis $\{e_1,e_2,e_3,e_4\}$ of the singular decomposition, we can write
$$u_1=\frac{1}{\sqrt{(1+\lambda^2)(1+\mu^2)}}\quad\text{and}\quad |u_2|=\frac{\lambda\mu}{\sqrt{(1+\lambda^2)(1+\mu^2)}}.$$
Note also that
$$\operatorname{Jac}(f):=\frac{*(f^*\Omega_N)}{*(\Id^*\Omega_M)}=\frac{u_2}{u_1}.$$
Moreover, the difference between $u_1$ and $|u_2|$ measures how far $f$ is from being area preserving. In particular:
\begin{eqnarray*}
u_1-|u_2|\ge 0 &\Leftrightarrow& f \text{ is area decreasing},\\
u_1-|u_2|> 0 &\Leftrightarrow& f \text{ is strictly area decreasing},\\
u_1-|u_2|= 0 &\Leftrightarrow& f \text{ is area preserving}.
\end{eqnarray*}

\subsection{The K{\"a}hler angles}
There are two natural complex structures associated to the product space $(M\times N,\gk)$, namely
$$J_1:=\pi^*_MJ_M-\pi^*_NJ_N\quad\text{and}\quad J_2:=\pi^*_MJ_M+\pi^*_NJ_N,$$
where $J_M, J_N$ are the complex structures on $M$ and $N$ defined by 
$$\Omega_M(\cdot\,,\cdot)=\gm(J_M\,\cdot\,,\cdot),\quad\Omega_N(\cdot\,,\cdot)=\gn(J_N\,\cdot\,,\cdot).$$
Chern and Wolfson \cite{chern} introduced a function which measures the deviation of the
tangent plane $\dF(T_xM)$ from a complex line of the space $T_{F(x)}(M\times N)$. More precisely, if
we consider $(M\times N,\gk)$ as a complex manifold with respect to $J_1$ then its corresponding
{\it K{\"a}hler angle} $a_1$ is given by the formula
$$\cos a_1=\varphi:=\gk\big(J_1\dF(v_1),\dF(v_2)\big)=u_1-u_2.$$
For our convenience we require that $a_1\in[0,\pi]$. Note that in general $a_1$
is not smooth at points where $\varphi=\pm 1.$
If there exists a point $x\in M$ such that $a_1(x)=0$ then $\dF(T_xM)$ is a complex line of $T_{F(x)}(M\times N)$ and
$x$ is called a {\it complex point} of $F$. If $a_1(x)=\pi$
then $\dF(T_xM)$ is an anti-complex line of $T_{F(x)}(M\times N)$ and $x$ is said {\it anti-complex point} of $F$. In the case
where $a_1(x)=\pi/2$, the point $x$ is called {\it Lagrangian point} of the map $F$. In this case $u_1=u_2$.
Similarly, if we regard the product manifold $(M\times N,\gk)$ as a K{\"a}hler manifold with respect to the complex structure $J_2$,
then its corresponding K{\"a}hler angle $a_2$ is defined by the formula
$$\cos a_2=\vartheta:=\gk\big(J_2\dF(v_1),\dF(v_2)\big)=u_1+u_2.$$

The graph $\Gamma(f)$ in the product K{\"a}hler manifold $(M\times N,\gk,J_i)$ is called
{\it symplectic} with respect to the K\"ahler form related to $J_i$, if the corresponding K{\"a}hler angle satisfies $\cos a_i>0$. Therefore a map $f$ is strictly area decreasing if and
only if its graph $\Gamma(f)$ is symplectic with respect to both K\"ahler forms. There are many 
interesting results on
symplectic mean curvature flow of surfaces in $4$-dimensional manifolds in the literature
(see for example the papers \cite{symp6,symp1,symp2,symp3,symp4,symp5}).

\subsection{Structure equations}
Around each point $x\in \Gamma(f)$ we choose an adapted local orthonormal 
frame $\{e_1,e_2;e_3,e_4\}$
such that $\{e_1,e_2\}$ is tangent and $\{e_3,e_4\}$ is normal to the graph. 
The components of $A$ are denoted as
$A^{\alpha}_{ij}:=\langle A(e_i,e_j),e_{\alpha}\rangle.$
Latin indices take values $1$ and $2$ while Greek indices take the values $3$ and $4$. For instance we write
the mean curvature vector in the form
$H=H^3e_3+H^4e_4. $
By \textit{Gau\ss' equation} we get
\begin{equation*}\label{Gausscurv}
2\sigma_{\gind}=2u^2_1\sigma_M+2u_{2}^2\sigma_N+\|H\|^2-\|A\|^2.
\end{equation*}
From the {\it Ricci equation} we see that the curvature $\sigma_n$ of the 
normal bundle of $\Gamma(f)$ is given
by the formula
\begin{equation*}
\sigma_n:=\rind^{\perp}_{1234}
=\rk_{1234}
+A^3_{11}A^4_{12}-A^3_{12}A^{4}_{11}+A^3_{12}A^4_{22}-A^3_{22}A^4_{12}.
\end{equation*}
The sum of the last four terms in the above formula is equal to minus the 
commutator $\sigma^{\perp}$ of the matrices $A^3=(A^3_{ij})$ and $A^4=(A^4_{ij})$, i.e.,
\begin{equation}\label{sigmaperp}
\sigma^{\perp}:=\langle [A^3,A^4]e_1,e_2\rangle=-A^3_{11}A^4_{12}+A^3_{12}A^{4}_{11}-A^3_{12}A^4_{22}+A^3_{22}A^4_{12}.\end{equation}
\section{A priori estimates for the Jacobians}
Let $M$ and $N$ be Riemann surfaces, $f:M\to N$ a smooth map and let $F:M \to M\times N$,
$F := \Id \times f,$ be the parametrization of the graph $\Gamma(f)$. Consider the family of immersions $F:M\times [0, T )\to M\times N$ satisfying
the mean curvature flow
\begin{equation*}
\left\{
\begin{array}{ll}
\dF_{(x,t)}(\partial_t)=H(x,t), &\\
F(x,0)=F(x),  &
\end{array}
\right.
\end{equation*}
where $(x,t)\in M\times [0, T )$, $H(x,t)$ is the mean curvature vector field at $x\in M$ of the immersion $F_t:M\to M\times N$
given by $F_t(\cdot):=F(\cdot,t)$
and $T$ is the maximal time of existence of the solution.
The compactness of $M$ implies that the evolving submanifolds stay graphs on an interval $[0,T_g)$ with $T_g\le T$.
This means that there exist smooth families of diffeomorphisms $\phi_t\in\operatorname{Diff}(M)$ and maps $f_t:M\to N$ such that
$$F_t\circ \phi_t = \Id\times f_t,$$
for any time $t\in [0,T_g)$.

\subsection{Evolution equations of first order quantities}


In the next lemma we recall the evolution equation of a parallel $2$-form on
the product manifold $M\times N$. The proofs can be found in \cite{wang1}.




\begin{lemma}\label{parallel}
Let $\Omega$ be a parallel $2$-form on the product manifold $M\times N$. Then, the function $u:=*(F^{\ast}\Omega)$ evolves in time under the equation
\begin{eqnarray*}
\partial_{t}u=\Delta u+\|A\|^2u-2\sum_{\alpha,\beta, k}A^{\alpha}_{ki}A^{\beta}_{kj}\Omega_{\alpha\beta}
+\sum_{\alpha}\big(\rk_{212\alpha}\Omega_{\alpha 2}+\rk_{121\alpha}\Omega_{1\alpha}\big)
\end{eqnarray*}
where  $\{e_1,e_2;e_3,e_4\}$ is an arbitrary adapted local orthonormal frame.
\end{lemma}

As a consequence of Lemma \ref{parallel}
we deduce the following:

\begin{lemma}\label{jacobians}
The functions $u_1$ and $u_2$ defined in section \ref{sec 2.4} satisfy the following coupled system of parabolic equations
\begin{eqnarray*}
\partial_{t}u_1-\Delta u_1\hspace{-5pt}&=&\hspace{-5pt}\|A\|^2u_1+2\sigma^{\perp} u_2+\sigma_M(1-u^2_1-u^2_2)u_1-2\sigma_Nu_1u^2_2,\\
\partial_{t}u_2-\Delta u_2\hspace{-5pt}&=&\hspace{-5pt}\|A\|^2u_2+2\sigma^{\perp} u_1+\sigma_N(1-u^2_1-u^2_2)u_2-2\sigma_Mu^2_1u_2.
\end{eqnarray*}
Moreover, $\varphi$ and $\vartheta$ satisfy the following system of equations
 \begin{eqnarray*}
\dt\varphi-\Delta\varphi\hspace{-6pt}&=&\hspace{-8pt}\big\{\|A\|
^2-2\sigma^{\perp}\big\}\varphi+\tfrac{1}
{2}\big\{\sigma_M(\varphi+\vartheta)+\sigma_N(\varphi-\vartheta)\big\}
(1-\varphi^2),\\
\dt\vartheta-\Delta\vartheta\hspace{-6pt}&=&\hspace{-8pt}\big\{\|A\|^2+2\sigma^{\perp}\big\}\vartheta+\tfrac{1}{2}\big\{\sigma_M(\varphi+\vartheta)-\sigma_N(\varphi-\vartheta)\big\}(1-\vartheta^2).
\end{eqnarray*}
In particular, if all the maps $f_t$ are area preserving, then the curvatures $\sigma_M$ and $\sigma_N$ necessarily must satisfy the relation $\sigma_M=\sigma_N\circ f_t$ for any $t\in[0,T_g)$.
\end{lemma}

\begin{proof}
The evolution equations of the functions $u_1$ and $u_2$ follow as an immediate consequence of Lemma \ref{parallel}.
Suppose now that each $f_t$ is an area preserving map. Then $\varphi=u_1-u_2=0$
in space and time. Combining the two
equations from above, we deduce that the curvatures of $M$ and $N$ are related by $\sigma_M=\sigma_N\circ f_t$,
and so $f_t$, $t\in[0,T_g)$, are even curvature preserving maps. This completes the proof of lemma.
\end{proof}

\subsection{Estimating the Jacobians.} We will give here several a priori estimates for the functions $u_1$,
$u_2$ and the K{\"a}hler angles.

\begin{lemma}
Let $f:M\to N$ be a smooth map between two complete Riemann surfaces, $M$ being compact. Then
the mean curvature flow of $\Gamma(f)$ stays graphical as long as it exists and the function
$u_2/u_1$ stays bounded.
\end{lemma}

\begin{proof}
From the first equation of Lemma \ref{jacobians} we deduce that there exists a time dependent and bounded function $h$ such that
$$\dt u_1-\Delta u_1\ge h\,u_1.$$
Then from the parabolic maximum principle we
get that $u_{1}(x,t) >0,$
for any $(x,t)\in M\times [0,T).$ Therefore, the solution remains graphical as long as the flow
exists.
\end{proof}

\begin{lemma}\label{l}
Let $f:(M,\gm)\to(N,\gn)$ be an area decreasing map. Suppose that the curvatures of the surfaces $(M,\gm)$ and $(N,\gn)$ satisfy
$\sigma:=\min\sigma_M\ge\sup\sigma_N.$
Then the following statements hold.
\begin{enumerate}[\rm(a)]
\item
The conditions $\operatorname{Jac}(f)\le 1$ or $\operatorname{Jac}(f)\ge -1$ are both preserved
under the mean curvature flow.
\medskip
\item
The area decreasing property is preserved under the flow.
\medskip
\item
If there is a point $(x_0,t_0)\in M\times (0,T_g)$
where $\operatorname{Jac}^2(f)=1$, then $\operatorname{Jac}^2(f)\equiv 1$ in space and time and $\sigma_M=\sigma=\sigma_N$.
\end{enumerate}
\end{lemma}

\begin{proof}
From Lemma \ref{jacobians}, we deduce that
\begin{eqnarray*}
{\dt}\varphi-\Delta \varphi&=&\big\{\|A\|^2-2\sigma^{\perp}+\sigma_N(1-\varphi^2)\big\}\varphi\\
&&+\frac{1}{2}(\sigma_M-\sigma_N)(\varphi+\vartheta)(1-\varphi^2).
\end{eqnarray*}
Note that the quantities $1-\varphi^2$ and $\varphi+\vartheta$ are positive. Hence, because of our curvature assumptions, the last line of the above
equality is positive. Thus, there exists a time dependent function $h$ such that
$$\dt \varphi-\Delta \varphi\ge h\,\varphi.$$
From the parabolic maximum principle we deduce that $\varphi$ stays positive in time. Moreover, from the strong parabolic
maximum principle it follows that if $\varphi$ vanishes somewhere, then it vanishes identically in space and time. Hence, the sign of $\varphi$
is preserved by the flow. Similarly we prove the results concerning $\vartheta$. This completes the proof. 
\end{proof}

Now we want to explore the behavior of the function
$$\rho=\varphi\vartheta=u^2_1-u^2_2$$
under the graphical mean curvature flow.
\begin{lemma}\label{crucial}
Suppose that $(M,\gm)$ and $(N,\gn)$ are complete Riemann surfaces with
$(M,\gm)$ being compact such that their
curvatures $\sigma_M$ and $\sigma_N$ are related by the inequality
$\sigma:=\min\sigma_M\ge\sup\sigma_N.$ Let $f:M\to N$ be a strictly
area decreasing map.
\begin{enumerate}[\rm(a)]
\item
If $\sigma\ge 0$, then there exists a positive constant $c_0$ such that
$$\rho\ge\frac{c_0e^{\sigma t}}{\sqrt{1+c^2_0e^{2\sigma t}}},$$
for any $(x,t)$ in space-time.
\smallskip
\item
If $\sigma<0$, then there exists a positive constant $c_0$ such that
$$\rho\ge\frac{c_0e^{2\sigma t}}{\sqrt{1+c^2_0e^{4\sigma t}}},$$
for any $(x,t)$ in space-time.
\end{enumerate}
\end{lemma}

\begin{proof}
From Lemma \ref{jacobians} we get,
\begin{equation*}
\dt\rho-\Delta\rho=2\rho\|A\|^2-2\langle\nabla\varphi,\nabla\vartheta\rangle+2(1-\rho)\sigma_Mu^2_1-2(1+\rho)\sigma_Nu^2_2.
\end{equation*}
Note that
\begin{eqnarray}
-2\rho\langle\nabla\varphi,\nabla\vartheta\rangle&+&\frac{1}{2}\|\nabla\rho\|^2
=\frac{1}{2}\big\{\|\nabla(\varphi\vartheta)\|^2-4\varphi\vartheta\langle\nabla\varphi,\nabla\vartheta\rangle\big\}\nonumber\\
&=&\frac{1}{2}\big\{\varphi^2\|\nabla\vartheta\|^2+\vartheta^2\|\nabla\varphi\|^2-2\varphi\vartheta\langle\nabla\varphi,\nabla\vartheta\rangle\big\}\nonumber\\
&\ge&\frac{1}{2}\big\{\|\varphi\nabla\vartheta\|-\|\vartheta\nabla\varphi\|\big\}^2.\nonumber
\end{eqnarray}
Since by assumption
$\sigma_M\ge\sigma\ge\sigma_N$, we deduce that
\begin{equation*}
\dt\rho-\Delta\rho\ge-\frac{1}{2\rho}\|\nabla\rho\|^2+2\sigma\rho(1-u^2_1-u^2_2).
\end{equation*}
One can algebraically check that
\begin{equation}\label{lm}
1-\rho^2\le 2(1-u^2_1-u^2_2)\le 2(1-\rho^2).
\end{equation}
Suppose at first that $\sigma\ge 0$. Then
$$\dt\rho-\Delta\rho\ge-\frac{1}{2\rho}\|\nabla\rho\|^2+\sigma\rho(1-\rho^2).$$
From the comparison maximum principle we obtain
$$\rho\ge\frac{c_0e^{\sigma t}}{\sqrt{1+c^2_0e^{2\sigma t}}},$$
where $c_0$ is a positive constant.

In the case where $\sigma<0$, from the equation \eqref{lm} we deduce that
$$\dt\rho-\Delta\rho\ge-\frac{1}{2\rho}\|\nabla\rho\|^2+2\sigma\rho(1-\rho^2),$$
from where we get the desired estimate.
\end{proof}

Let us state here the following auxiliary result which will be used later for several estimates. The proof is straightforward.

\begin{lemma}\label{evzeta}
Let $f:(M,\gm)\to(N,\gn)$ be an area decreasing map. Let $\eta$ be a positive 
smooth function depending on $\rho$ and let $\zeta$ be the function given 
by
$$\zeta:=\log\eta(\rho).$$
Then
\begin{eqnarray*}
\dt\zeta-\Delta\zeta&=&\frac{2\rho\eta_{\rho}}{\eta}\|A\|^2
+\frac{\eta_{\rho}}{\eta}\Big\{-2\langle\nabla\varphi,\nabla\vartheta\rangle
+\frac{1}{2\rho}\|\nabla\rho\|^2\Big\}\\
&&
-\frac{1}
{2\rho\eta^2}\big\{\eta\eta_{\rho}+2\rho\eta\eta_{\rho\rho}
-\rho\eta_{\rho}^2\big\}\|\nabla\rho\|^2+\frac{1}{2}\|\nabla\zeta\|^2\\
&&+\frac{2\eta_{\rho}}{\eta}\big\{(1-\rho)\sigma_Mu^2_1-(1+\rho)\sigma_Nu^2_2\big\}.
\end{eqnarray*}
\end{lemma}

\section{A priori decay estimates for the mean curvature}
We will show in this section that under our curvature assumptions, in the strictly area decreasing
case, the norm of the mean curvature vector stays uniformly bounded as long as the flow exists.

\begin{lemma}\label{evmean}
Let $f:(M,\gm)\to(N,\gn)$ be an area decreasing map. Suppose that the curvatures of $M$ and $N$ satisfy
$\sigma:=\min\sigma_M\ge\sup\sigma_N.$
Let $\delta:[0,T)\to\real{}$ be a positive increasing real function and $\tau$ the time dependent function given by
$$\tau:=\log\big(\delta\|H\|^2+\varepsilon\big),$$
where $\varepsilon$ is a non-negative number.
Then,
\begin{eqnarray*}
\dt\tau-\Delta\tau&\le&\frac{2\delta}{\delta\|H\|^2+\varepsilon}\|H\|^2\|A\|^2+\frac{\delta'}{\delta\|H\|^2+\varepsilon}\|H\|^2\\
&&+\frac{2\delta}{\delta\|H\|^2+\varepsilon}\|H\|^2\sigma_M(1-u^2_1-u^2_2)+\frac{1}{2}\|\nabla\tau\|^2.
\end{eqnarray*}
\end{lemma}
\begin{proof}
Recall from \cite[Corollary 3.8]{sm} that the squared norm $\|H\|^2$ of the mean curvature vector evolves in time under the equation
\begin{eqnarray*}
\dt\|H\|^2-\Delta\|H\|^2&=&2\|A^H\|^2-2\|\nabla^{\perp}H\|^2\\
&&+2\rk(H,e_1,H,e_1)+ 2\rk(H,e_2,H,e_2),
\end{eqnarray*}
where $\{e_1,e_2\}$ is a local orthonormal frame with respect to $\gind$. Using the special frames introduced in subsection \ref{singular} we see that
\begin{eqnarray*}
&&\rk(H,e_1,H,e_1)+\rk(H,e_2,H,e_2)\nonumber\\
&&\quad\quad=\{H^2\}^2\rk(e_4,e_1,e_4,e_1)+\{H^1\}^2\rk(e_3,e_1,e_3,e_1)\nonumber\\
&&\quad\quad=\{H^4\}^2u^2_1(\mu^2\sigma_M+\lambda^2\sigma_N)+\{H^3\}^2u^2_1(\lambda^2\sigma_M+\mu^2\sigma_N)\nonumber\\
&&\quad\quad=\sigma_Mu^2_1(\lambda^2+\mu^2)\|H\|^2
-(\sigma_M-\sigma_N)u^2_1\big[\lambda^2\{H^4\}^2+\mu^2\{H^3\}^2\big]\nonumber\\
&&\quad\quad\le\sigma_M(1-u^2_1-u^2_2)\|H\|^2.
\end{eqnarray*}
Note that from Cauchy-Schwarz inequality
$$\|A^H\|\le\|A\|\|H\|.$$
Moreover, observe that at points where the mean curvature vector is non-zero,
from Kato\rq{s} inequality, we have that
$$\|\nabla^{\perp}H\|^2\ge\|\nabla\|H\|\|^2.$$
Consequently, at points where the norm $\|H\|$ of the mean curvature is not zero the following 
inequality holds
$$
\dt\|H\|^2-\Delta\|H\|^2\le-2\|\nabla\|H\|\|^2+2\|A\|^2\|H\|^2
+2\sigma_M(1-u^2_1-u^2_2)\|H\|^2.
$$
Now let us compute the evolution equation of the function $\tau$. We have,
\begin{eqnarray*}
\dt\tau-\Delta\tau&=&\frac{\delta(\dt\|H\|^2-\Delta\|H\|^2)}{\delta\|H\|^2+\varepsilon}
+\frac{\delta^2\|\nabla\|H\|^2\|^2}{(\delta\|H\|^2+\varepsilon)^2}
+\frac{\delta'\|H\|^2}{\delta\|H\|^2+\varepsilon}\\
&\le&-\frac{2\delta}{\delta\|H\|^2+\varepsilon}\|\nabla\|H\|\|^2
+\frac{\delta^2}{(\delta\|H\|^2+\varepsilon)^2}\|\nabla\|H\|^2\|^2\\
&&+\frac{2\delta}{\delta\|H\|^2+\varepsilon}\|H\|^2\|A\|^2
+\frac{\delta'}{\delta\|H\|^2+\varepsilon}\|H\|^2\\
&&+\frac{2\delta}{2\delta\|H\|^2+\varepsilon}\|H\|^2\sigma_M(1-u^2_1-u^2_2).
\end{eqnarray*}
Note that
$$-\frac{2\delta}{\delta\|H\|^2+\varepsilon}\|\nabla\|H\|\|^2
+\frac{1}{2}\frac{\delta^2}{(\delta\|H\|^2+\varepsilon)^2}\|\nabla\|H\|^2\|^2\le 0.$$
Therefore,
\begin{eqnarray*}
\dt\tau-\Delta\tau&\le&\frac{1}{2}\|\nabla\tau\|^2
+\frac{2\delta}{\delta\|H\|^2+\varepsilon}\|H\|^2\|A\|^2\\
&&+\frac{\delta'}{\delta\|H\|^2+\varepsilon}\|H\|^2
+\frac{2\delta}{\delta\|H\|^2+\varepsilon}\|H\|^2\sigma_M(1-u^2_1-u^2_2),
\end{eqnarray*}
and this completes the proof.
\end{proof}

\begin{theorem}\label{estmean}
Let $f:(M,\gm)\to(N,\gn)$ be an area decreasing map, where $M$ is compact 
and $N$ a complete Riemann surface. Suppose that the curvatures of $M$ and $N$ satisfy
$\sigma:=\min\sigma_M\ge\sup\sigma_N.$
Then the following statements hold.
\begin{enumerate}[\rm(a)]
\item
There exist
a positive time independent constant $C$ such that
$$\|H\|^2\le C,$$
as long as the flow exists.
\smallskip
\item
If $\sigma\ge 0$, the following improved
decay estimate holds
$$\|H\|^2\le C t^{-1},$$
where $C$ is again a positive constant.
\end{enumerate}
\end{theorem}

\begin{proof}
Consider the time dependent function $\Theta$ given by
$$\Theta:=\log\frac{\delta\|H\|^2+\varepsilon}{\rho},$$
where $\delta$ is a positive increasing function. Making use of the estimate
$$\|H\|^2\le 2\|A\|^2$$
and from the evolution equations of
Lemma \ref{evzeta} and Lemma \ref{evmean} we deduce that
\begin{eqnarray*}
\dt\Theta-\Delta\Theta&\le&\frac{1}{2}\langle\nabla\Theta,\nabla\tau+\nabla\rho\rangle\\
&&+\frac{\delta'\|H\|^2-\varepsilon\|H\|^2
-2\varepsilon\sigma(1-u^2_1-u^2_2)}{\delta\|H\|^2+\varepsilon}.
\end{eqnarray*}
Choosing $\delta=1$ and $\varepsilon=0$, we obtain that
$$\dt \Theta-\Delta\Theta\le\frac{1}{2}\langle\nabla\Theta,\nabla\tau+\nabla\rho\rangle.$$
From the maximum principle the norm $\|H\|$ remains uniformly bounded in time regardless of the sign of the constant $\sigma$. 
In the case where $\sigma\ge 0$, choosing
$\varepsilon=1$ and $\delta=t$, we 
deduce that $\Theta$ remains uniformly bounded in time which gives the desired decay
estimate for $H$.
\end{proof}

\section{Blow-up analysis and convergence}

\subsection{Cheeger-Gromov compactness for metrics}
Let us recall here the basic notions and definitions.
Fore more details see the books \cite[Chapter 5]{morgan}, \cite[Chapter 3]{chow} and \cite[Chapter 9]{andrews}.

\begin{definition}[$C^{\infty}$-convergence] Let $(E,\pi,\Sigma)$ be a vector
bundle endowed
with a Riemannian metric
$\gind$ and a metric connection $\nabla$ and suppose that
$\{\xi_k\}_{k\in\natural{}}$ is a sequence of sections of $E$.
Let $\Omega$ be an open subset of $\Sigma$ with compact closure $\bar{\Omega}$
in $\Sigma$. Fix a natural number $p\ge 0$. We say that $\{\xi_k\}_{k\in\natural{}}$
converges in $C^p$ to $\xi_{\infty}\in\Gamma(E|_{\bar\Omega})$, if for every
$\varepsilon>0$ there exists $k_0=k_0(\varepsilon)$ such that
$$\sup_{0\le\alpha\le p}\sup_{x\in\bar{\Omega}}\big\|\nabla^{\alpha}
(\xi_k-\xi_{\infty})\big\|<\varepsilon$$
whenever $k\ge k_0.$ We say that $\{\xi_k\}_{k\in\natural{}}$ converges in $C^{\infty}$
to $\xi_{\infty}\in\Gamma(E|_{\bar\Omega})$ if $\{\xi_k\}_{k\in\natural{}}$ converges in 
$C^{p}$ to $\xi_{\infty}\in\Gamma(E|_{\bar\Omega})$ for any $p\in\natural{}$.
\end{definition}

\begin{definition}[$C^{\infty}$-convergence on compact sets] Let $(E,\pi,\Sigma)$ be a
vector bundle  endowed with a Riemannian metric $\gind$ and a metric connection $\nabla$. 
Let $\{U_n\}_{n\in\natural{}}$ be an exhaustion of $\Sigma$
and $\{\xi_{k}\}_{k\in\natural{}}$ be a sequence of sections of $E$ defined on open
sets $A_{k}$ of $\Sigma$. We say that $\{\xi_{k}\}_{k\in\natural{}}$ converges
smoothly on
compact sets to $\xi_{\infty}\in\Gamma(E)$ if:
\begin{enumerate}[\rm (a)]
\item
For every $n\in\natural{}$ there exists
$k_0$ such that $\bar{U}_n\subset A_k$ for all natural numbers $k\ge k_0$.
\smallskip
\item
The sequence
$\{\xi|_{\bar{U}_k}\}_{k\ge k_0}$ converges in $C^{\infty}$ to the restriction of the section
$\xi_{\infty}$ on $\bar{U}_k$.
\end{enumerate}
\end{definition}

In the next definitions we recall the notion of smooth Cheeger-Gromov convergence of sequences
of Riemannian manifolds.

\begin{definition}[Pointed manifolds]
A pointed Riemannian manifold $(\Sigma,\gind,x)$ is a Riemannian manifold
$(\Sigma,\gind)$
with a choice of origin or base point $x\in \Sigma$. If the metric $\gind$ is complete, we say that
$(\Sigma,\gind,x)$ is a complete pointed Riemannian manifold.
\end{definition}

\begin{definition}[Cheeger-Gromov smooth convergence] \label{54}
A sequence of complete pointed
Riemannian manifolds $\{(\Sigma_k,\gind_k,x_k)\}_{k\in\natural{}}$ smoothly converges 
in the sense of Cheeger-Gromov to a 
complete 
pointed  Riemannian manifold $(\Sigma_{\infty},\gind_{\infty},x_{\infty})$, if there exists:
\begin{enumerate}[\rm (a)]
\item
An exhaustion $\{U_k\}_{k\in\natural{}}$ of $\Sigma_{\infty}$ with $x_{\infty}\in U_k$, for all 
$k\in\natural{}$.
\medskip
\item
A sequence of diffeomorphisms
$\Phi_k:U_k\to\Phi_k(U_k)\subset\Sigma_k$ with
$\Phi_k(x_{\infty})=x_k$ and such that $\{\Phi_k^*\gind_k\}_{k\in\natural{}}$ smoothly converges
in $C^{\infty}$ to $\gind_{\infty}$ on compact sets in $\Sigma_{\infty}$.
\end{enumerate}
The family $\{(U_k,\Phi_k)\}_{k\in\natural{}}$ is called a family of convergence pairs
of the sequence $\{(\Sigma_k,\gind_k,x_k)\}_{k\in\natural{}}$ with respect to the limit
$(\Sigma_{\infty},\gind_{\infty},x_{\infty})$.
\end{definition}

In the sequel, when we say {\sl smooth convergence}, we will always mean smooth convergence
in the sense of Cheeger-Gromov.

The family of convergence pairs is not unique. However, two such families
$\{(U_k,\Phi_k)\}_{k\in\natural{}}$,
$\{(W_k,\Psi_k)\}_{k\in\natural{}}$ are equivalent in the sense that
there exists an isometry $\mathcal{I}$ of the limit $(\Sigma_{\infty},\gind_{\infty},x_{\infty})$
such that, for every compact subset $K$ of $\Sigma_{\infty}$ there exists a natural
number $k_0$ such that for any natural $k\ge k_0$:
\begin{enumerate}[\rm (a)]
\item
the mapping $\Phi^{-1}_k\circ\Psi_k$ is well defined over $K$ and
\medskip
\item
the sequence $\{\Phi^{-1}_k\circ\Psi_k\}_{k\ge k_0}$ smoothly converges to $\mathcal{I}$
on $K$.
\end{enumerate}
In fact, the limiting pointed Riemannian manifold $(\Sigma_{\infty},\gind_{\infty},x_{\infty})$ 
of the Definition \ref{54} is unique up to isometries (see \cite[Lemma 5.5]{morgan}).

\begin{definition}
A complete Riemannian manifold $(\Sigma,\gind)$ is said to have bounded geometry, if the following
conditions are satisfied:
\begin{enumerate}[\rm (a)] 
\item
For any integer $j\ge 0$ there exists a uniform constant $C_j$ such that 
$\|\nabla^j\operatorname{R}(\gind)\|\le C_j.$
\smallskip
\item 
The injectivity radius satisfies $\operatorname{inj}_{\gind}(\Sigma)>0$.
\end{enumerate}
\end{definition}
The following proposition is standard and will be useful in the proof of the long time existence of the graphical mean
curvature flow.

\begin{proposition}\label{CG1}
Let $(\Sigma,\gind)$ be a complete Riemannian manifold with bounded geometry. Suppose that $\{a_k\}_{k\in\natural{}}$ is an increasing sequence of real numbers that tends to $+\infty$
and let $\{x_k\}_{k\in\natural{}}$ be a sequence of points on $\Sigma.$ Then, the sequence
$(\Sigma,a^2_k\gind,x_k)$ smoothly subconverges to the standard euclidean space $(\real{m},\gind_{\operatorname{euc}},0)$.
\end{proposition}

%

We will use the following definition of uniformly bounded geometry for a sequence of pointed Riemannian manifolds.

\begin{definition}
We say that a sequence $\{(\Sigma_k,\gind_k,x_k)\}_{k\in\natural{}}$ of complete pointed Riemannian manifolds
has uniformly bounded geometry if the following conditions are satisfied:
\begin{enumerate}[\rm (a)]
\item
For any $j\ge 0$ there exists a uniform constant $C_j$ such that for each $k\in\natural{}$ it 
holds
$\|\nabla^j\operatorname{R}(\gind_k)\|\le C_j.$
\item
\smallskip
There exists a uniform constant $c_0$ such $\operatorname{inj}_{\gind_k}(\Sigma_k)\ge c_0>0.$
\end{enumerate}
\end{definition}
In the next result we state the Cheeger-Gromov compactness theorem for sequences of
complete pointed Riemannian manifolds. The version that we present here is due to Hamilton (see for
example \cite{hamilton3} or \cite[Chapters 3 \& 4]{chow}). 
\begin{theorem}[Cheeger-Gromov compactness]
Let $\{(\Sigma_k,\gind_k,x_k)\}_{k\in\natural{}}$ be a sequence
of complete pointed Riemannian manifolds with uniformly bounded geometry.
Then, the sequence $\{(\Sigma_k,\gind_k,x_k)\}_{k\in\natural{}}$ subconverges smoothly to
a complete pointed Riemannian manifold $(\Sigma_{\infty},\gind_{\infty},x_{\infty})$.
\end{theorem}
\begin{remark}
Due to an estimate from Cheeger, Gromov and Taylor \cite{cheeger}, the above compactness
theorem still holds under the weaker assumption that the injectivity radius is uniformly bounded from below by a positive constant only along
the base points $\{x_k\}_{k\in\natural{}}$, thereby avoiding the assumption of the uniform lower bound for
$\operatorname{inj}_{\gind_k}(\Sigma_k)$.
\end{remark}
\subsection{Convergence of immersions}

Let us begin our exposition with the geometric limit of a sequence of immersions.

\begin{definition}[Convergence of isometric immersions]
Suppose that $\{(L_k,\operatorname{h}_k,y_k)\}_{k\in\natural{}}$ is a sequence of pointed Riemannian manifolds
and $\{F_k\}_{k\in\natural{}}$ a sequence of isometric immersions
$F_k:(\Sigma_k,\gind_k)\to(L_k,\operatorname{h}_k)$ such that
$F_k(x_k)=y_k$, where $\{\Sigma_k\}_{k\in\natural{}}$ is a family of manifolds and $\{x_k\}_{\in\natural{}}$
a sequence such that $x_k\in\Sigma_k$ for any $k\in\natural{}$. We say that the sequence $\{F_k\}_{k\in\natural{}}$ converges
smoothly to an isometric immersion $F_{\infty}:
(\Sigma_{\infty},\gind_{\infty},x_{\infty})\to(L_{\infty},\operatorname{h}_{\infty},y_{\infty})$ if the following conditions are
satisfied:
\begin{enumerate}[\rm (a)]
\item
The sequence $\{(\Sigma_k,\gind_k,x_k)\}_{k\in\natural{}}$ smoothly converges to the pointed
Riemannian manifold $(\Sigma_{\infty},\gind_{\infty},x_{\infty}).$
\smallskip
\item
The sequence $\{(L_k,\operatorname{h}_k,y_k)\}_{k\in\natural{}}$ smoothly converges to the
pointed Riemannian manifold $(L_{\infty},\operatorname{h}_{\infty},y_{\infty}).$
\smallskip
\item
If $\{(U_k,\Phi_k)\}_{k\in\natural{}}$ is a family of convergence pairs of the sequence
$\{(\Sigma_k,\gind_k,x_k)\}_{k\in\natural{}}$ and $\{(W_k,\Psi_k)\}_{k\in\natural{}}$ is a 
family of convergence pairs of the sequence
$\{(L_k,\operatorname{h}_k,y_k)\}_{k\in\natural{}}$ then,
for each $k\in\natural{}$, it holds $F_k\circ\Phi_k(U_k)\subset\Psi_k(W_k)$ and
$\Psi_k^{-1}\circ F\circ\Phi_k$ smoothly converges
to $F_{\infty}$ on compact sets.
\end{enumerate}
\end{definition}

The following result
holds true (see for example \cite[Corollary 2.1.11]{cooper} or \cite[Theorem 2.1]{chen2}).

\begin{lemma}
Suppose that $(L,\operatorname{h})$ is a complete Riemannian manifold with
bounded geometry. Then for any $C>0$ there exists a positive constant $r>0$ such
that
$\operatorname{inj}_{\gind}(\Sigma)>r$
for any isometric immersion $F:(\Sigma,\gind)\to (L,\operatorname{h})$
such that the norm $\|A_F\|$ of its second fundamental form
satisfies $\|A_F\|\le C$.
\end{lemma}
The last lemma and the Cheeger-Gromov compactness theorem allow us to obtain a compactness theorem in the category of 
sequences of immersions (see for instance \cite[Theorem 2.0.12]{cooper}).

\begin{theorem}[Compactness for immersions]
Let $\{(\Sigma_k,\gind_k,x_k)\}_{k\in\natural{}}$,
$\{(L_k,\operatorname{h}_k,y_k)\}_{k\in\natural{}}$ be sequences of complete
Riemannian manifolds with dimension $m$ and $l$ respectively. Let $F_k:
(\Sigma_k,\gind_k)\to(L_k,\operatorname{h}_k)$ be a family of isometric immersions with 
$F_k(x_k)=y_k$. Assume that:
\begin{enumerate}[\rm (a)]
\item
Each $\Sigma_k$ is compact.
\medskip
\item
The sequence 
$\{(L_k,\operatorname{h}_k,y_k)\}_{k\in\natural{}}$ has uniformly
bounded geometry.
\medskip
\item
For each integer $j\ge 0$ there exists
a uniform constant $C_j$ such that
$$\|(\nabla^{F_k})^jA_{F_k}\|\le C_j,$$
for any $k\in\natural{}$. Here $A_{F_k}$ stands for the second fundamental form
of the immersion $F_k$.
\end{enumerate}
Then the sequence of immersions $\{F_k\}_{k\in\natural{}}$ subconverges smoothly
to a complete isometric immersion $F_{\infty}:(\Sigma_{\infty},\gind_{\infty},x_{\infty})\to
(L_{\infty},\operatorname{h}_{\infty},y_{\infty})$.
\end{theorem}

\subsection{Modeling the singularities}
The next theorem shows how one can built smooth singularity models for the
mean curvature flow by rescaling properly around points where the second
fundamental form attains its maximum. The proof relies on the compactness theorem
of Cheeger-Gromov 
and on the compactness theorem for immersions. For more details see
\cite[Theorem 2.4 and Proposition 2.5]{chen}.

\begin{theorem}[Blow-up limit]\label{blow}
Let $\Sigma$ be a compact manifold and
$F:\Sigma\times [0,T)\to (L,\operatorname{h})$ be a solution of mean curvature 
flow, where $L$ is a Riemannian manifold with bounded geometry and $T\le\infty$ is the
maximal time of existence. Suppose that there exists a sequence of points 
$\{(x_k,t_k)\}_{k\in\natural{}}$
in $\Sigma\times[0,T)$ with $\lim t_k=T$ and such that the sequence 
$\{a_k\}_{k\in\natural{}}$, where
$$a_k:=\max_{(x,t)\in\Sigma\times[0, t_k]}\|A(x,t)\|=\|A(x_k,t_k)\|,$$
tends to infinity. Then:
\begin{enumerate}[\rm (a)]
\item
The maps
$F_k:\Sigma\times[-a^{2}_kt_k,0]
\to(L,a^2_k\operatorname{h})$, $k\in\natural{}$,
given by
$$F_k(x,s):=F_{k,s}(x):=F(x,{s}/{a^{2}_k}+t_k),$$
form a sequence of mean curvature flow solutions. Moreover,
$$\|A_{F_{\infty}}\|\le 1\quad\text{and}\quad \|A_{F_{\infty}}(x_{\infty},0)\|=1.$$
\item
For any fixed $s\le 0$, the sequence $\{(\Sigma,F_{k,s}^*(a^2_k\operatorname{h}),x_k)\}_{k\in\natural{}}$
smoothly subconverges to a complete Riemannian manifold
$(\Sigma_{\infty},\gind_{\infty},x_{\infty})$ that does not depend on the choice of $s$. Moreover, the
sequence of pointed manifolds
$\{(L,a^2_k\operatorname{h},F_k(x_k,s))\}_{k\in\natural{}}$ smoothly subconverges
to the standard euclidean space $(\real{l},\gind_{\operatorname{euc}},0)$.
\medskip
\item
There is a mean curvature flow $F_{\infty}:\Sigma_{\infty}\times(-\infty,0]\to\real{l}$,
such that for each fixed time $s\le 0$, the sequence $\{F_{k,s}\}_{k\in\natural{}}$
smoothly subconverges to $F_{\infty, s}$. This convergence is uniform with respect to the parameter $s$.
Additionally,
$$\|A_{F_{\infty}}\|\le 1\quad\text{and}\quad \|A_{F_{\infty}}(x_{\infty},0)\|=1.$$
\item
If $\dim\Sigma=2$ and $H_{F_{\infty}}=0$, then the limiting Riemann surface $\Sigma_{\infty}$ has finite total curvature.
In the matter of fact, $\Sigma_{\infty}$
is conformally diffeomorphic to a compact Riemann surface minus a finite number of points and is of parabolic
type.
\end{enumerate}
\end{theorem}

\subsection{Long time existence and convergence}
Now we shall prove that under our assumptions the graphical mean curvature flow exists for all
time and smoothly converges.

\begin{theorem}\label{blow1}
Let $(M,\gm)$ and $(N,\gn)$ be Riemann surface as in Theorem \ref{thma}
and let $f:M\to N$ be a
strictly area decreasing map. Evolve the graph of $f$ under the mean curvature flow.
Then the norm of the second fundamental form of the evolved graphs stays uniformly
bounded in time and the graphs smoothly converge to a
minimal surface $M_{\infty}$ of $M\times N$.
\end{theorem}
\begin{proof}
Suppose to the contrary that $\|A\|$ is not uniformly bounded. Then
there exists a sequence of points 
$\{(x_k,t_k)\}_{k\in\natural{}}$
in $M\times[0,T)$ with $\lim t_k=T$ with
$$a_k:=\max_{(x,t)\in M \times[0, t_k]}\|A(x,t)\|=\|A(x_k,t_k)\|,$$
and such that the sequence $\{a_k\}_{k\in\natural{}}$
tends to infinity. Now perform scalings as in Theorem \ref{blow}.
A direct computation shows that the mean curvature vector $H_{k}$ of
$F_k$ is related to the mean curvature $H$ of $F$ by
$$H_{k}(x,s)=a^{-2}_kH(x,s/a^{2}_k+t_k),$$
for any $(x,s)\in M\times[-a^{2}_kt_k,0].$
Let $F_{\infty}:\Sigma_{\infty}\times(-\infty,0]\to\real{4}$ be the blow-up flow of
Theorem \ref{blow}.
Since the norm $\|H\|$ is uniformly bounded and the convergence is smooth,
we deduce
that $F_{\infty}:\Sigma_{\infty}\to\real{4}$ must be a complete
minimal immersion. We will distinguish two cases:

{\bf Case A.} Suppose at first that $\sigma\ge 0$. In this case, due to Lemma \ref{crucial} the 
singular values of the
time dependend map $f_t:M\to N$ remain uniformly bounded in time and
$\operatorname{Jac}(f_t)$ remains uniformly bounded by $1$. Recall now from Theorem
\ref{blow}(d) that the Riemann surface $\Sigma_{\infty}$ is parabolic. Consequently,
any positive superharmonic function must be constant. Now observe that the 
corresponding K\"ahler  angles $\varphi_{\infty}$, $\vartheta_{\infty}$ of $F_{\infty}$ with 
respect to the complex structures
$J=(J_{\real{2}},-J_{\real{2}})$ and $J_2=(J_{\real{2}},J_{\real{2}})$
of $\real{4}$ are strictly positive. Moreover, as in Lemma \ref{jacobians} we get that
\begin{eqnarray}
&&\Delta\varphi_{\infty}+\{\|A_{F_{\infty}}\|^2-2\sigma^{\perp}_{F_{\infty}}\}\varphi_{\infty}=0\label{eq philim}\\
&&\Delta\vartheta_{\infty}+\{\|A_{F_{\infty}}\|^2+2\sigma^{\perp}_{F_{\infty}}\}\vartheta_{\infty}=0,\label{eq thetalim}
\end{eqnarray}
where $-\sigma^{\perp}_{F_{\infty}}$ is the normal curvature of $F_{\infty}$.
Note that from equation (\ref{sigmaperp}) one can easily derive the inequalities
$$\|A_{F_{\infty}}\|^2\pm2\sigma^{\perp}_{F_{\infty}}\ge 0.$$
Hence, $\varphi_{\infty}$ and $\vartheta_{\infty}$ must be non-zero positive constants
and $\|A_{F_{\infty}}\|=0$. This contradicts the fact that there is a point where 
$\|A_{F_{\infty}}\|=1$. This completes Case A.

{\bf Case B.} Let us investigate the case where $\sigma<0$. If $T<\infty$, then by the
estimate done in Lemma \ref{crucial}(b) we see that
the singular values of $f_t$ are uniformly bounded
from above and we can
proceed exactly in the same way as in Case A. So let us assume from now on that 
$T=\infty$. In this case the biggest singular
value of the evolving maps $f_t:M\to N$ might tend to $+\infty$ as $t$ tends to $+\infty$ while $\operatorname{Jac}(f_t)<1$
for any $t\ge 0$. 

There are only to possibilities concerning the Jacobian of the limiting minimal immersion.
The first option is
$$\lim_{t\to+\infty}\operatorname{Jac}^2(f_t)(x)<1,$$
or, equivalently,
$$\lim_{t\to+\infty}u_1^2(x,t)>\lim_{t\to+\infty} u^2_2(x,t)\ge 0, $$
for any $x\in M$. Since the convergence is smooth, proceeding as in Case A we deduce that the immersion $F_{\infty}$ is flat, which is absurd.

The next possibility is that there is a point $x_0\in M$ such that
$$\lim_{t\to+\infty}\operatorname{Jac}(f_t)(x_0)=\pm 1.$$
Suppose that the above limit is equal to $1$. The case where the limit is $-1$ is treated
in a similar way. 

{\bf Claim.} {\it The limiting surface is Lagrangian and $\lim_{t\to\infty}\operatorname{Jac}(f_t)=1$ uniformly.}

{\it Proof of the claim.}
 Let $\varphi_{\infty}$ be the K\"ahler angle of $F_{\infty}$ with respect to the complex structure
$J=(J_{\real{2}},-J_{\real{2}})$
of $\real{4}$, where by $J_{\real{2}}$ we denote the standard complex structure of $\real{2}$. Since
the convergence is smooth we deduce that $\varphi_{\infty}\ge 0$ and that there exists a point where $\varphi_{\infty}$
becomes zero. Recall that $\varphi_{\infty}$ satisfies the partial differential equation
$$\Delta\varphi_{\infty}+\{\|A_{F_{\infty}}\|^2-2\sigma^{\perp}_{F_{\infty}}\}\varphi_{\infty}=0,$$
where $-\sigma^{\perp}_{F_{\infty}}$ is the normal curvature of $F_{\infty}$.
Since there is a point where $\varphi_{\infty}$ vanishes, from the strong maximum principle we deduce that
$\varphi_{\infty}$ must be identically zero. Consequently,
$F_{\infty}:\Sigma_{\infty}\to\real{4}$ must be a complete minimal Lagrangian immersion.
From the relations $\varphi_\infty=(u_1)_\infty-(u_2)_\infty$, $\vartheta_\infty=(u_1)_\infty+(u_2)_\infty$ and equations (\ref{eq philim}), (\ref{eq thetalim}) we then see that $(u_1)_\infty=(u_2)_\infty=const.$ If this constant is
non-zero, then certainly $\lim_{t\to\infty}\operatorname{Jac}(f_t)=1$ uniformly, because
$\operatorname{Jac}(f_t)=u_2/u_1$.

On the other hand, if that constant is zero, then
$\varphi_\infty=\vartheta_\infty=0$ and the equations for the gradients of $\varphi_{\infty}$, $\vartheta_{\infty}$,
$$\|\nabla\varphi_\infty\|^2
=(1-\varphi_\infty^2)\big\{\big[(A_{F_\infty})^3_{11}+(A_{F_\infty})^4_{12}\big]^2+\big[(A_{F_\infty})^3_{12}-(A_{F_\infty})^4_{11}\big]^2\big\}$$
$$\|\nabla\vartheta_\infty\|^2
=(1-\vartheta_\infty^2)\big\{\big[(A_{F_\infty})^3_{11}-(A_{F_\infty})^4_{12}\big]^2+\big[(A_{F_\infty})^3_{12}+(A_{F_\infty})^4_{11}\big]^2\big\}$$
imply that $A_{F_\infty}$ vanishes identically, which is a contradiction to the fact that $\|A_{F_{\infty}}\|$ attains the value $1$ somewhere. Hence the limit
of the Jacobian is $1$ everywhere. This completes the proof of the claim.

It is well known (see for example \cite{chen1} or \cite{lagrangian}) that
minimal Lagrangian surfaces in $\complex{2}$ are holomorphic curves with respect 
to one of the complex structures of $\complex{2}$. In the matter of
fact (see for instance \cite{aiyama}) we can explicitly locally
reparametrize the minimal Lagrangian immersion $F_{\infty}$ in the form
$$F_{\infty}=\frac{1}{\sqrt{2}}e^{i\beta/2}\big(\mathcal{F}_1-i\overline{\mathcal{F}_2},\mathcal{F}_2+i\overline{\mathcal{F}_1}\big),$$
where $\beta$ is a constant and $\mathcal{F}_1$, $\mathcal{F}_2:\mathbb{D}\subset\complex{}\to\complex{}$ are holomorphic functions defined in a simply connected domain $\mathbb{D}$ such that
$$|(\mathcal{F}_1)_z|^2+|(\mathcal{F}_2)_z|^2>0.$$
The Gau{\ss} image of $F_{\infty}$ lies in the slice $\mathbb{S}^{2}\times\{(e^{i\beta},0)\}$
of $\mathbb{S}^2\times\mathbb{S}^2$. In the matter of fact all the information on the Gau{\ss} image of $F_{\infty}$
is encoded in the map $\mathcal{G}:\mathbb{D}\to\mathbb{S}^{2}=\complex{}\cup\{\infty\}$ given by
$$\mathcal{G}={(\mathcal{F}_1)_z}/{(\mathcal{F}_2)_z}.$$
In the case where the immersion $F_{\infty}$ was the graph of an area preserving map $f:\complex{}\to\complex{}$, then
$$\mathcal{F}_1=({z+i\overline{f}})/2,\quad\mathcal{F}_2
=(-i\overline{z}+f)/2\quad\text{and}\quad |f_z|^2-|f_{\bar{z}}|^2=1.$$
Therefore
$$\mathcal{G}={(\mathcal{F}_1)_z}/{(\mathcal{F}_2)_z}=(1-i{f}_{\bar z})/{f_z}.$$
A straightforward computation shows that
$$|\mathcal{G}|^2=\frac{\big|1+i\overline{f_{\bar z}}\big|^2}{|f_z|^2}
=\frac{1+|f_{\bar z}|^2
+i\big(\overline{f_{\bar z}}-f_{\bar z}\big)}{1+|f_{\bar z}|^2}
=1+\frac{2\operatorname{Im}(f_{\bar z})}{1+|f_{\bar z}|^2}\le 2.$$
In this case the image of $\mathcal{G}$ is contained in a bounded subset of $\complex{}\cup\{\infty\}$.
On the other hand, recall that
$F_{\infty}$ arises as a smooth limit
of graphical surfaces generated by maps whose Jacobians uniformly approaches the value $1$.
Consequently, also in the general case, the Gau{\ss} image of $F_{\infty}$ omits an open subset of $\mathbb{C}\cup\{\infty\}$.
But then, due to a result of Osserman \cite[Theorem 1.2]{osserman} the immersion $F_{\infty}$ must be
flat, which is a contradiction. This completes Case B.

Since the norm of the second fundamental form is uniformly bounded in time, the graphical mean curvature flow exists for all time
and, moreover, due to the general convergence theorem of Simon
\cite[Theorem 2]{simon}, it smoothly converges to a compact 
minimal surface $M_{\infty}$ of the product space $M\times N$. This completes the proof of the 
theorem.
\end{proof}

\begin{remark}
In the case where $F_{\infty}(\Sigma_{\infty})$ is an entire minimal graph, in the proof of the above theorem,
one could use the Bernstein type theorems proved by Hasanis, Savas-Halilaj
and Vlachos in \cite{hasanis1,hasanis2} to show flatness of $F_{\infty}$.
\end{remark}

\section{Proof of Theorem B}

In this section we will prove the decay estimates claimed in Theorem $B$. Let us start by proving the following auxiliary lemma.

\begin{lemma}\label{evsecond}
Let $f:(M,\gm)\to(N,\gn)$ be an area decreasing map, where $M$ and $N$ are Riemann surfaces as
in Theorem \ref{thma}. Suppose that $\sigma:=\min\sigma_M>0$. Consider the time dependent function $g$ given by
$$g:=\log(t\|A\|^2+1).$$
Then $g$ satisfies the following inequality
\begin{eqnarray*}
&&\dt g-\Delta g\le 3\|A\|^2+\frac{1}{2}\|\nabla g\|^2+C(1+\sqrt{t})\sqrt{1-\rho^2},
\end{eqnarray*}
where $C$ is a positive real constant.
\end{lemma}
\begin{proof}
Recall from \cite[Proposition 7.1]{wang1} that
\begin{eqnarray*}
\dt\|A\|^2&=&\Delta\|A\|^2-2\|\nabla^{\perp}A\|^2\nonumber\\
&+&2\sum_{i,j,k,l}\Big\{\sum_{\alpha}A^{\alpha}_{ij}A^{\alpha}_{kl}\Big\}^2
+2\sum_{\alpha,\beta,i,j}
\Big\{\sum_{k}(A^{\alpha}_{ik}A^{\beta}_{jk}
-A^{\beta}_{ik}A^{\alpha}_{jk}\big)\Big\}^2\nonumber\\
&+&4\sum_{\alpha,i,j,k,l}\Big\{A^{\alpha}_{ij}A^{\alpha}_{kl}
-\delta_{kl}\sum_{p}A^{\alpha}_{ip}A^{\alpha}_{jp}\Big\}\tilde{R}_{kilj}\nonumber\\
&+&2\sum_{\alpha,\beta,i,j,k}\Big\{4A^{\alpha}_{jk}A^{\beta}_{ik}
\tilde{R}_{\alpha\beta ji}
+A^{\alpha}_{jk}A^{\beta}_{jk}\tilde{R}_{\alpha i\beta i}\Big\}\nonumber\\
&+&2\sum_{\alpha, i,j,k}A^{\alpha}_{jk}
\Big\{\big({\nabla}_{i}\tilde{R}\big)_{\alpha jki}
+\big({\nabla}_{k}\tilde{R}\big)_{\alpha iji} \Big\},
\end{eqnarray*}
where  the indices are with respect to an arbitrary adapted local orthonormal frame $\{e_1,e_2;e_3,e_4\}$.

From \cite[Theorem 1]{anmin}, we have that
$$
2\sum_{i,j,k,l}\Big\{\sum_{\alpha}A^{\alpha}_{ij}A^{\alpha}_{kl}\Big\}^2
+2\sum_{\alpha,\beta,i,j}
\Big\{\sum_{k}(A^{\alpha}_{ik}A^{\beta}_{jk}
-A^{\beta}_{ik}A^{\alpha}_{jk}\big)\Big\}^2\le 3\|A\|^4.
$$
Consider now the term
\begin{eqnarray*}
\mathcal{A}_1:&=&4\sum_{i,j,k,l,\alpha}\Big\{A^{\alpha}_{ij}A^{\alpha}_{kl}-\delta_{kl}\sum_{p} A^{\alpha}_{ip}A^{\alpha}_{jp}\Big\}\rk_{kilj}\\
&&+2\sum_{i,j,k,\alpha,\beta}\Big\{4A^{\alpha}_{jk}A^{\beta}_{ik}\rk_{\alpha\beta ji}+A^{\alpha}_{jk}A^{\beta}_{jk}\rk_{\alpha i\beta i}\Big\}.
\end{eqnarray*}
In terms of the frame fields introduced in subsection \ref{singular}, we get that
\begin{eqnarray*}
\mathcal{A}_1&=&-4\big(\sigma_Mu^2_1+\sigma_Nu^2_2\big)\big\{\|A_{11}-A_{22}\|^2
+4\|A_{12}\|^2\big\}\\
&&+2\|A^3\|^2u^2_1\big(\lambda^2\sigma_M+\mu^2\sigma_N\big)
+2\|A^4\|^2u^2_1\big(\mu^2\sigma_M+\lambda^2\sigma_N\big)\\
&&-16u_1|u_2|(\sigma_M+\sigma_N)\sigma^{\perp}\\
&\le&-4u^2_2\sigma_N\big\{\|A_{11}-A_{22}\|^2
+4\|A_{12}\|^2\big\}+2u^2_1(\lambda^2+\mu^2)\|A\|^2\sigma_M\\
&&-16u_1|u_2|(\sigma_M+\sigma_N)\sigma^{\perp}.
\end{eqnarray*}
Since the evolving graphs are area decreasing, we see that
$$2u^2_2=2\lambda^2\mu^2u^2_1\le2\lambda\mu u^2_1\le (\lambda^2+\mu^2)u^2_1=1-u^2_1-u^2_2.$$
Additionally,
$$2u_1|u_2|=2\lambda\mu u^2_1\le u^2_1(\lambda^2+\mu^2)\le 1-u^2_1-u^2_2.$$
Because the Riemann surfaces $M$ and $N$ have bounded geometry we deduce that there exists
a constant $C_1$ such that
\begin{equation*}
\mathcal{A}_1\le C_1(1-u^2_1-u^2_2)\|A\|^2.
\end{equation*}
Denote by $\mathcal{A}_2$ the term
\begin{equation*}
\mathcal{A}_2:=2\sum_{\alpha, i,j,k}A^{\alpha}_{jk}
\Big\{\big({\nabla}_{i}\tilde{R}\big)_{\alpha jki}
+\big({\nabla}_{k}\tilde{R}\big)_{\alpha iji} \Big\}.
\end{equation*}
Similarly we deduce that there exists a constant $K_2$ such that
$$\mathcal{A}_2\le K_2\|A\|u^2_1\big(\lambda+\mu+\lambda^3\mu+\lambda\mu^3+\lambda^2\mu^2\big).$$
Because by assumption the map $f$ is area decreasing and since $u_1<1$, we obtain
\begin{eqnarray*}
\mathcal{A}_2&\le& K_2\|A\|u^2_1\big(\lambda+\mu+\lambda^2+\mu^2+\lambda\mu\big)\\
&\le&K_2\|A\|u^2_1\left\{\sqrt{2(\lambda^2+\mu^2)}+\frac{3}{2}(\lambda^2+\mu^2)\right\}\\
&\le&K_2\|A\|\left\{\sqrt{2(1-u^2_1-u^2_2)}+\frac{3}{2}\big(1-u^2_1-u^2_2\big)\right\}\\
&\le&\Bigl(\sqrt{2}+\frac{3}{2}\Bigr)K_2\|A\|\sqrt{1-u^2_1-u^2_2}.
\end{eqnarray*}
Going back to the evolution equation of $\|A\|^2$ we deduce that there are
constants $C_1$ and $C_2$ such that
\begin{eqnarray*}
\dt\|A\|^2\hspace{-4pt}&-&\hspace{-4pt}\Delta\|A\|^2\le
-2\|\nabla^{\perp}A\|^2+3\|A\|^4\\
&+&C_1(1-u^2_1-u^2_2)\|A\|^2+C_2\sqrt{1-u^2_1-u^2_2}\,\|A\|.
\end{eqnarray*}
Let us compute now the evolution equation of $g$. By straightforward computations we have,
\begin{eqnarray*}
\dt g&=&\Delta g+\frac{t}{t\|A\|^2+1}\big(\dt \|A\|^2-\Delta\|A\|^2\big)+\frac{\|A\|^2}{t\|A\|^2+1}+\|\nabla g\|^2\\
&\le&\Delta g+\|\nabla g\|^2-\frac{2 t}{t\|A\|^2+1}\|\nabla^{\perp}A\|^2+\frac{3 t\|A\|^2+1}{ t\|A\|^2+1}\|A\|^2\\
&&+C_1(1-u^2_1-u^2_2)\frac{t\|A\|^2}{t\|A\|^2+1}
+C_2\sqrt{1-u^2_1-u^2_2}\,\frac{ t\|A\|}{t\|A\|^2+1}\\
&\le&\Delta g+\frac{1}{2}\|\nabla g\|^2-\frac{2 t}{ t\|A\|^2+1}\|\nabla\|A\|\|^2+\frac{1}{2}\|\nabla g\|^2+3\|A\|^2\\
&&+C_1(1-u^2_1-u^2_2)\frac{t\|A\|^2}{t\|A\|^2+1}
+C_2\sqrt{1-u^2_1-u^2_2}\,\frac{t\|A\|}{t\|A\|^2+1}.
\end{eqnarray*}
Consequently,
\begin{eqnarray*}
\dt g&\le&\Delta g+\frac{1}{2}\|\nabla g\|^2+3\|A\|^2\\
&&+C_1\sqrt{1-u^2_1-u^2_2 }\,\frac{t\|A\|^2}{t\|A\|^2+1}
+C_2\sqrt{1-u^2_1-u^2_2 }\,\frac{\sqrt{t}\sqrt{t}\|A\|}{t\|A\|^2+1}\\
&\le&\Delta g+\frac{1}{2}\|\nabla g\|^2+3\|A\|^2+C\big(1+\sqrt{t}\big)\sqrt{1-\rho^2}
\end{eqnarray*}
where $C$ is a positive constant.
This completes the proof.
\end{proof}

In the following result we give the decay estimates for the norm of the second fundamental form.

\begin{theorem}\label{thm2}
Let $(M,\gm)$ be a complete Riemann surfaces as in Theorem A and let $f:M\to N$ be a
strictly area decreasing map. Let $\sigma:=\min\sigma_M$. Then, the following statements hold true:
\begin{enumerate}[\rm (a)]
\item
If $\sigma>0$, then there exists a constant $C$ such that
$$\|A\|^2\le{C}{t^{-1}}.$$
\item
If $\sigma=0$, then there exists a constant $C$ such that
$$\int_M\|A\|^2\,\Omega_M\le Ct^{-1}.$$
\end{enumerate}
\end{theorem}
\begin{proof}
Recall from Theorem \ref{blow1} that $\|A\|$ is uniformly bounded and that the flow exists for all time.
Let us now consider the following cases depending on the sign of $\sigma$.

(a) Suppose at first that $\sigma>0$. Consider the function $\Phi$ given by the formula
$$\Phi:=g-\zeta=\log\frac{t\|A\|^2+1}{\eta(\rho)},$$
where $\eta(\rho)$ is a positive increasing function depending on $\rho$ that will be determined
later.
From Lemma \ref{evzeta} and Lemma \ref{evsecond}, the evolution equation of $\Phi$ is
\begin{eqnarray*}
\dt\Phi&\le&\Delta\Phi+\frac{3\eta-2\rho\eta_{\rho}}{\eta}\|A\|^2+C\big(1+\sqrt{t}\big)\sqrt{1-\rho^2}\\
&&+\frac{1}{2}\langle\nabla\Phi,\nabla g+\nabla\zeta\rangle
+\frac{1}{2\rho\eta^2}\big(\eta\eta_{\rho}
+2\rho\eta\eta_{\rho\rho}-\rho\eta_{\rho}^2\big)\|\nabla\rho\|^2\\
&&-\frac{2\eta_{\rho}}{\eta}\left\{(1-\rho)\sigma_Mu^2_1-(1+\rho)\sigma_Nu^2_2\right\}.
\end{eqnarray*}
Hence,
\begin{eqnarray*}
\dt\Phi&\le&\Delta\Phi+\frac{3\eta-2\rho\eta_{\rho}}{\eta}\|A\|^2+C\big(1+\sqrt{t}\big)\sqrt{1-\rho^2}\\
&&+\frac{1}{2}\langle\nabla\Phi,\nabla g+\nabla\zeta\rangle
+\frac{1}{2\rho\eta^2}\big(\eta\eta_{\rho}
+2\rho\eta\eta_{\rho\rho}-\rho\eta_{\rho}^2\big)\|\nabla\rho\|^2\\
&&-\frac{2\sigma\rho\eta_{\rho}}{\eta}(1-u^2_1-u^2_2).
\end{eqnarray*}
Since $\sigma>0$, we get that
\begin{eqnarray*}
\dt\Phi&\le&\Delta\Phi+\frac{3\eta-2\rho\eta_\rho}{\eta}\|A\|^2+C\big(1+\sqrt{t}\big)\sqrt{1-\rho^2}\\
&&+\frac{1}{2}\langle\nabla\Phi,\nabla g+\nabla\zeta\rangle
+\frac{1}{2\rho\eta^2}\big(\eta\eta_{\rho}
+2\rho\eta\eta_{\rho\rho}-\rho\eta^2_{\rho}\big)\|\nabla\rho\|^2.
\end{eqnarray*}
Let us choose for $\eta$ the smooth function given by
$$\eta(\rho):=\left(-\frac{1}{3}+\sqrt{\rho}\,\right)^2.$$
Since the flow exists for all time, from Lemma \eqref{crucial}(a) and from the fact that $\rho\le 1$ we see that $\rho$ tends to $1$ uniformly
as time tends to infinity. Thus, there exists a $t_0>0$ such that $\eta(\rho)>0$ for all $t\in[t_0,+\infty)$. Moreover, for this
choice of $\eta$, we see that
$$\frac{3\eta-2\rho\eta_{\rho}}{\eta}=\frac{3\big(\sqrt{\rho}-1\big)}{3\sqrt{\rho}-1}\le 0.$$
By making again use of Lemma \ref{crucial}(a), we deduce
that there exists a positive constant $c_0$ such that
\begin{eqnarray*}
\dt\Phi-\Delta\Phi-\frac{1}{2}\langle\nabla\Phi,\nabla g+\nabla\zeta\rangle&\le& C\big(1+\sqrt{t}\big)\sqrt{1-\rho^2}
\le\frac{C\big(1+\sqrt{t}\big)}{\sqrt{1+c^2_0 e^{2\sigma t}}}\\
&\le&\frac{C}{c_0}\big(1+\sqrt{t}\big)e^{-\sigma t}.
\end{eqnarray*}
Let $y$ be the solution of the ordinary differential equation
$$y'(t)=\frac{C}{c_0}\big(1+\sqrt{t}\big)e^{-\sigma t},\quad y(0)=\max_{x\in M} \Phi(x,0).$$
From the parabolic maximum principle it follows that $\Phi(x,t)\le y(t)$ for any
$(x,t)\in M\times[0,\infty)$. Therefore
$\Phi$ is uniformly bounded because the solution $y$ is bounded. 
This implies that there exists a constant, which we denote again by $C$,
such that
$$t\|A\|^2\le C.$$
(b) Suppose that $\sigma=0$. Denote by $\Omega_{\gind(t)}$ the volume forms of the induced
metrics. Because of the formula
$$\dt\left(\int_{M}\Omega_{\gind(t)}\right)=-\int_{M} \|H\|^2\Omega_{\gind(t)}\le 0,$$
we obtain that
$$\int_M\Omega_{\gind(t)}\le \int_M\Omega_{\gind(0)}=\operatorname{constant}.$$
Now from Theorem \ref{estmean}(b) it follows that there is a non-negative constant $C$ such that
$$\int_M \|H\|^2\Omega_{\gind(t)}\le\frac{C}{t}\int_M\Omega_{\gind(t)}
\le\frac{C}{t}\int_M\Omega_{\gind(0)}.$$
Recall that due to our assumptions we have that
$u^2_2\le u^2_1\le 1$ and $\min\sigma_M\ge 0\ge \sup\sigma_N.$
Moreover, recall that
$$\Omega_{\gind(t)}=\sqrt{(1+\lambda^2)(1+\mu^2)}\Omega_M=\frac{1}{\,u_1}\Omega_M.$$
From the Gau{\ss} equation \eqref{Gausscurv} and the Gau{\ss}-Bonnet formula we get
\begin{eqnarray*}
&&\int_M\|A\|^2\Omega_{\gind(t)}
=\int_M\|H\|^2\Omega_{\gind(t)}\\
&&\quad\quad\quad+2\int_M\big(\sigma_Mu^2_1+\sigma_Nu^2_2\big)\Omega_{\gind(t)}
-2\int_M \sigma_{\gind(t)}\Omega_{\gind(t)}\\
&&\quad\quad\quad\le 2\int_M\sigma_Mu^2_1\Omega_{\gind(t)}
-2\int_M \sigma_{\gind(t)}\operatorname{Vol}_{\gind(t)}
+\int_M\|H\|^2\operatorname{Vol}_{\gind(t)}\\
&&\quad\quad\quad\le 2\int_M\sigma_Mu_1\Omega_{\gind(t)}
-2\int_M \sigma_{\gind(t)}\Omega_{\gind(t)}
+\int_M\|H\|^2\operatorname{Vol}_{\gind(t)}\\
&&\quad\quad\quad\le2\int_M\sigma_M\Omega_M
-2\int_M \sigma_{\gind(t)}\Omega_{\gind(t)}
+\int_M\|H\|^2\Omega_{\gind(t)}\\
&&\quad\quad\quad=\int_M\|H\|^2\Omega_{\gind(t)}.
\end{eqnarray*}
From the above inequality we get the decay estimate of the
$L^2$-norm of $\|A\|$.
This completes the proof of part (b).
\end{proof}

From Theorems \ref{blow1}, \ref{estmean} and \ref{thm2} we immediately obtain 
the results stated in Theorem B.

\section{Proof of the Theorem A}
Suppose that $(M,\gm)$ and $(N,\gn)$ are two Riemann surfaces satisfying the assumptions of Theorem \ref{thma}
and let $\sigma=\min\sigma_M$.
Let $f:M\to N$ be an area decreasing map. Then the property
of being area decreasing is preserved by the flow and, moreover, the flow remains graphical
for all time. In the matter of fact, there are two options: either the map $f$ is immediately
deformed
into a strictly area decreasing one or each map $f_t$, $t\in[0,T)$, is area preserving, $N$ is 
compact and the curvatures of $M$ and $N$ are constant and satisfy
$\sigma_M=\sigma=\sigma_N.$
The area preserving case is completely solved in \cite{wang2} and \cite{smoczyk1}. 
Thus, it remains to examine the case where $f$ becomes strictly area decreasing. In this case,
from Theorem \ref{blow1} we know that the graphical mean curvature flow, independently of the sign of $\sigma$, 
smoothly converges to a minimal surface $M_{\infty}$.

Suppose that $\sigma>0$. In this case the flow is smoothly
converging to a graphical minimal surface $M_{\infty}=\Gamma(f_{\infty})$ of
$M\times N$. Due to Theorem \ref{thm2}(b), $M_{\infty}$ must be totally geodesic
and $f_{\infty}$ is a constant map.

Assume now that $\sigma=0$. As in the previous case we have smooth convergence 
of the flow to a minimal graphical surface
$M_{\infty}=\Gamma(f_{\infty})$ of $M\times N$, where $f_{\infty}$ is a strictly area 
decreasing map. From the integral inequality of Theorem \ref{thm2}(b) we deduce that
$$\int_M\|A_{\infty}\|^2=0.$$
Consequently, $M_{\infty}$ must be a totally geodesic graphical surface.

{\bf Acknowledgments:} This work was initiated during the research visit of both authors at the Max-Planck-Institut f\"ur Mathematik
in den Naturwissenschaften Leipzig in August 2014. The authors would like to express their gratitude to
J{\"u}rgen Jost and the Institute for the
excellent research conditions and the hospitality.


\end{document}